\documentclass[12pt]{amsart}
\usepackage{amssymb,latexsym,esint}
\usepackage{enumerate}
\usepackage{float}
\usepackage{frame, amssymb,latexsym,esint, physics}
\usepackage{enumerate}
\usepackage{float}
\usepackage{import}

\usepackage{mathtools}
\usepackage{amsfonts}
\usepackage{amssymb}
\usepackage{amsmath}
\usepackage{graphicx}
\usepackage{graphicx,color}
\usepackage[normalem]{ulem}
\usepackage{todonotes}
\usepackage{mathrsfs}

\usepackage{import}
\usepackage{xifthen}
\usepackage{pdfpages}
\usepackage{transparent}
\usepackage{yfonts}

\numberwithin{equation}{section}
\makeatletter
\@namedef{subjclassname@2020}{%
  \textup{2020} Mathematics Subject Classification}
\makeatother

\usepackage{mdframed, amsmath,amstext,amsthm,amsxtra,mathtools,amssymb}

\usepackage[bookmarksopen,bookmarksdepth=3,colorlinks,citecolor=red,pagebackref,hypertexnames=true]{hyperref}
\usepackage[backrefs,msc-links,nobysame,initials,non-sorted-cites]{amsrefs}

\usepackage[normalem]{ulem}
\usepackage{lmodern}

\usepackage{lmodern}

\usepackage[backrefs,msc-links,nobysame,initials]{amsrefs}

\allowdisplaybreaks

\newtheorem{thm}{Theorem}

\theoremstyle{definition}

\theoremstyle{remark}

\usepackage[noframe]{showframe}
\usepackage{framed}
\usepackage{lipsum}
 {\endMakeFramed}
\definecolor{shadecolor}{gray}{0.75}

\usepackage{amsmath,amstext,amsthm,amssymb,amsxtra}
\usepackage{mathtools}
\usepackage{enumerate}

\usepackage{amsfonts}
\usepackage{amssymb}
\usepackage{amsmath}
\usepackage{graphicx}
\usepackage{graphicx,color}
\usepackage[normalem]{ulem}
\usepackage{todonotes}

\def\norm#1.#2.{\lVert#1\rVert_{#2}}
\def\Norm#1.#2.{\bigl\lVert#1\bigr\rVert_{#2}}
\def\NOrm#1.#2.{\Bigl\lVert#1\Bigr\rVert_{#2}}
\def\NORm#1.#2.{\biggl\lVert#1\biggr\rVert_{#2}}
\def\NORM#1.#2.{\Biggl\lVert#1\Biggr\rVert_{#2}}

\begin{document}
\title[$K_\sigma$ sets associated to trees not satisfying a separation condition]{Remarks on the construction of $K_\sigma$ sets associated to trees not satisfying a separation condition}

\author[Paul  Hagelstein]{Paul Hagelstein}
\address{P. H.: Department of Mathematics, Baylor University, Waco, Texas 76798}
\email{\href{mailto:paul_hagelstein@baylor.edu}{paul\_hagelstein@baylor.edu}}
\thanks{P. H. is partially supported by Simons Foundation grant MP-TSM-00002046.}

\author[Blanca Radillo-Murguia]{Blanca Radillo-Murguia}
\address{B. R.-M.: Department of Mathematics, Baylor University, Waco, Texas 76798}
\email{\href{mailto:blanca_radillo1@baylor.edu}{blanca\_radillo1@baylor.edu}}

\author{Alex Stokolos}
\address{A. S.: Department of Mathematical Sciences, Georgia Southern University, Statesboro, Georgia 30460}
\email{\href{mailto:astokolos@GeorgiaSouthern.edu}{astokolos@GeorgiaSouthern.edu}}


\maketitle

\begin{abstract}  $K_\sigma$ sets involving sticky maps $\sigma$ have been used in the theory of differentiation of integrals to probabilistically construct Kakeya-type sets that imply  certain types of  directional maximal operators are unbounded on $L^p(\mathbb{R}^2)$ for all $1 \leq p < \infty$.  We indicate limits to this approach by showing that, given $\epsilon > 0$ and a natural number $N$, there exists a tree $\mathcal{T}_{N, \epsilon}$ of finite height  that is lacunary of order $N$ but such that, for \emph{every} sticky map $\sigma: \mathcal{B}^{h(\mathcal{T}_{N, \epsilon})} \rightarrow \mathcal{T}_{N, \epsilon}$, one has $|K_{\sigma} \cap ((1,2) \times \mathbb{R})| \geq 1 - \epsilon$.  
\end{abstract}
\mbox{}

\section{Introduction}

Let $\Omega$ be a nonempty subset of  $[0,1]$.   Associated to $\Omega$ is the \emph{directional maximal operator} $M_\Omega$ acting on measurable functions on $\mathbb{R}^2$ defined by
$$M_\Omega f(x) = \sup_{x \in R}\frac{1}{|R|}\int_R |f|\;,$$
where the supremum is taken over all rectangles in $\mathbb{R}^2$ containing $x$ with longest side having slope in $\Omega$.  

If $\Omega = [0,1]$, the maximal operator $M_\Omega$ is unbounded on $L^p(\mathbb{R}^2)$ for all $1 \leq p < \infty$ \cite{ busemannfeller, nikodym}.   If $\Omega$ is the lacunary set $\{2^{-j} : j \in \mathbb{N}\}$, then $M_\Omega$ is bounded on $L^p(\mathbb{R}^2)$ for $1 < p \leq \infty$ \cite{cf2, nsw, stromberg}.  More generally, if $\Omega$ is $N$-lacunary, then $M_\Omega$ is bounded on $L^p(\mathbb{R}^2)$ for $1 < p \leq \infty$ \cite{sjsj}.

In \cite{batemankatz}, Bateman and Katz utilized probabilistic methods involving \emph{sticky maps} to show that the maximal operator $M_{\mathcal{C}}$ associated to the ternary Cantor set $\mathcal{C}$ is unbounded on $L^p(\mathbb{R}^2)$ for $1 \leq p < \infty$.  Bateman subsequently announced a result in \cite{bateman} that the maximal operator $M_\Omega$ is bounded on $L^p(\mathbb{R}^2)$ for all $1 < p \leq \infty$ if and only if $\Omega$ is a finite union of sets of finite lacunary order.   Bateman's clever argument involved using probabilistic methods to show that, if $\Omega$ were not of finite lacunary order, then for every $N \in \mathbb{N}$ there would exist a sticky map $\sigma$ and associated sets $K_\sigma = K_{\sigma, 1} \cup K_{\sigma, 2}$ (where the $K_{\sigma, j}$ are of positive measure and  with a structure that we will detail in the next section) such that $|K_{\sigma, 1}| \gtrsim (\ln N) |K_{\sigma,2}|$, with $M_\Omega \chi_{K_{\sigma, 2}} \gtrsim \frac{1}{2}$ on $K_{\sigma, 1}$.  Unfortunately, we recently discovered a subtle gap in the proof of this statement in the case that the set $\Omega$, although not of finite lacunary order, fails to satisfy a \emph{separation condition}, although the proof does hold with some minor modification if the separation condition is satisfied \cite{hrs2024}.   We now recognize that  there exist certain non-separated non-finite lacunary sets $\Omega$ for which the desired $K_\sigma$ sets simply do not exist.   The purpose of this paper is to show this is the case.

In Section 2 we will define the appropriate terminology, largely following that of Bateman in \cite{bateman}.   In Section 3 we will construct a set of directions $\Omega \subset [0,1]$ for which the desired $K_\sigma$ sets do not exist.  In Section 4 we will suggest further directions for research in this area.
\\

We wish to thank the referee for the helpful comments and suggestions regarding this  paper.

\section{Terminology}

In this section we, largely following the terminology and setup of Bateman in \cite{bateman}, define sets $K_\sigma$ associated to sticky maps $\sigma$ mapping a truncated binary tree to itself.

We first define the binary tree $\mathscr{B}$.   We fix a vertex $v_0$, called the origin, and define $\mathscr{B}_0 = \{v_0\}$.    Suppose $\mathscr{B}_n$ has been defined.   To each vertex $v \in \mathscr{B}_n$ we associate two  new vertices $c_0(v)$ and $c_1(v)$ and define $$\mathscr{B}_{n+1} = \cup_{v \in \mathscr{B}_n}\left\{c_0(v), c_1(v)\right\}\;.$$   We define the binary tree $\mathscr{B}$
to be the graph with vertices in $\cup_{n=0}^\infty \mathscr{B}_{n}$ and edges connecting a vertex $v$ with each of its children $c_0(v)$ and $c_1(v)$.  We say the vertices in $\mathscr{B}_n$ are of height $n$.   If the vertex $v$ is of height $n$, we may write this as $h(v) = n$.

Given a vertex $v \in \mathscr{B}$, we define a ray $R$ rooted at $v$ to be an ordered set of vertices $v_1 = v, v_2, v_3, \ldots$ such that $v_{j+1}$ is a child of $v_j$ for $j=1,2,\ldots$.   Given a subtree $\mathscr{T}$ of $\mathscr{B}$ and a vertex $v \in \mathscr{T}$, we set $\mathfrak{R}_\mathscr{T}(v)$ to be the collection of all rays rooted at $v$ with vertices in $\mathscr{T}$.   

If $u \in R$ for some $R \in \mathfrak{R}_{\mathscr{T}}(v)$, we say $u$ is a descendant of $v$ or that $v$ is an ancestor of $u$.

Given a subtree $\mathscr{T}$ of $\mathscr{B}$ and  $h \in \mathbb{N}$, by $\mathscr{T}^h$ we denote the induced subtree of $\mathscr{T}$ associated to its vertices of height less than or equal to $h$.

Given a subtree $\mathscr{T}$ of $\mathscr{B}$, we say a vertex $v \in \mathscr{T}$ splits, or we say $v$ is a splitting vertex, if $v$ has two children in $\mathscr{T}$.   We define the splitting number $\textup{split}(R)$ of a ray $R$ in $\mathscr{T}$ to be the number of splitting vertices in $\mathscr{T}$ on $R$.   The splitting number of a vertex $v$ with respect to a tree $\mathscr{S}$ rooted at $v$ is defined as
$$\textup{split}_\mathscr{S}(v) = \min_{R \in \mathfrak{R}_\mathscr{S}(v)} \textup{split}(R)\;,$$
and the splitting number of $v$ is defined as
$$\textup{split}(v) = \sup_\mathscr{S} \textup{split}_\mathscr{S}(v)\;,$$
where the supremum is taken over all subtrees $\mathscr{S}$ of $\mathscr{T}$ rooted at $v$.  For a tree $\mathscr{T}$, we set
$$\textup{split}(\mathscr{T}) = \sup_{v \in \mathscr{T}} \textup{split}(v)\;,$$
where the supremum is taken over all the vertices $v$ in $\mathscr{T}$.

A tree $\mathscr{T} \subset \mathscr{B}$ is said to be lacunary of order $0$ if it consists of a single ray (possibly truncated to be of finite height) rooted at the origin of $\mathscr{B}$.  For $N \geq 1$,  $\mathscr{T}$ is said to be lacunary of order $N$ if all of the splitting vertices of $\mathscr{T}$ lie on a lacunary tree of order $N-1$ and moreover that $\mathscr{T}$ is not lacunary of order $N-1$.  

Let $\mathscr{T} \subset \mathscr{B}$ be lacunary of order $N$.   We say $\mathscr{T}$ is pruned provided, for every ray $R \in \mathfrak{R}_\mathscr{T}(v_0)$ and every $j = 1, 2, \ldots, N$, $R$ contains exactly one vertex $v_j$ such that $\textup{split } v_j = j$.
 
Let $v$ be a vertex in $\mathscr{B}$ of height $k$.  Let $(j_1, \ldots, j_k)$ be a sequence of $0$'s and $1$'s such that, letting $v_0$ denote the origin, $v$ lies on the ray $v_0, v_1, v_2, \ldots, v_k = v, \ldots$ in $\mathscr{B}$ such that $v_i = c_{j_i}(v_{i-1})$ for $i = 1, \ldots, k$.  For notational convenience, we will on occasion denote $v$ by the $(k + 1)$-string $0 j_1 \cdots j_k$, with $v_0$ itself being denoted simply by the 1-string 0.

Let $\sigma: \mathscr{B}^M \rightarrow \mathscr{B}^M$.  $\sigma$ is said to be a \emph{sticky map} if  $h(\sigma(v)) = h(v)$ for all $v \in \mathscr{B}^M$ and $\sigma(u)$ is an ancestor of $\sigma(v)$ whenever $u,v \in \mathscr{B}^M$ and $u$ is an ancestor of $v$.
\\

To each $\sigma: \mathscr{B}^M \rightarrow \mathscr{B}^M$ we may construct a set $K_\sigma \subset \mathbb{R}^2$ as follows. 
\\

For every $(M+1)$-string $v = 0 j_1 \cdots j_M$ consisting of $0$'s and $1$'s, let $(k_1, \ldots, k_M)$ be such that $\sigma(v) = 0 k_1 \cdots k_M$.   Let $\rho_v$ denote the parallelogram with vertices at the points $\left(0, \sum_{i=1}^M 2^{-i}j_i \right)$, $\left(0, 2^{-M} + \sum_{i=1}^M 2^{-i}j_i \right)$, 
$\left(2, \sum_{i=1}^M 2^{-i}j_i + 2\sum_{i=1}^M 2^{-i}k_i\right)$, $\left(2,  2^{-M} + \sum_{i=1}^M2^{-i}j_i  +   2\sum_{i=1}^M 2^{-i}k_i\right)$.  Define the set $K_\sigma \subset \mathbb{R}^2$ by
$$K_\sigma = \bigcup_{v \in \mathscr{B} : h(v) = M} \rho_v\;.$$

Note that if $v$ is a vertex of height $M$ in $\mathscr{B}^M$, the parallelogram $\rho_v \subset \mathbb{R}^2$  has a left-hand side that is an interval on the $y$-axis of  length $2^{-M}$ whose lowest point has a $y$-coordinate  in $[0,1]$ with binary expansion given by $v$.
\\

The primary result of this paper is the following.

\begin{thm}\label{t1}
Let $N \in \mathbb{N}$ and $\epsilon > 0$.   There exists a pruned tree $\mathscr{P}$ of finite height that is lacunary of order $N$ such that, for \emph{every} sticky map $\sigma: \mathscr{B}^{h(\mathscr{P})} \rightarrow \mathscr{P}$, we have
$$\left|K_\sigma \cap ([1,2] \times \mathbb{R})\right| \geq 1 - \epsilon\;.$$
\end{thm}

In contrasting this result with Claim 7(B) of \cite{bateman}, it is helpful to recognize that, as indicated in \cite {hrs2024}, the proof of Claim 7(B) implicitly relies on an assumption that $\mathscr{P}$ satisfies a separation condition.   Theorem \ref{t1} indicates what can happen if such a separation condition is not satisfied.

\section{Proof of Theorem \ref{t1}}

\begin{proof}[Proof of Theorem \ref{t1}]
Let $N \in \mathbb{N}$ and $\epsilon > 0$.  Let $k_1, \ldots, k_N$ be a sequence of natural numbers, all greater than 2, such that
$$8[2^{-k_1} + \cdots + 2^{-k_N}]  < \epsilon\;.$$

For $1 \leq i \leq N$, let $a_i$ denote the $k_i$-string $0111\cdots1$ and $b_i$ denote the $k_i$-string $1000\cdots0$.  For example, if  $k_1 =5$, then $a_1 = 01111$ and $b_1 = 10000$ . Let $\mathscr{P}$ be the pruned tree consisting of all of the vertices of the form $0s_1 \cdots s_n$ for $1 \leq n \leq N$, where each $s_i$ is either the string $a_i$ or $b_i$, together with the ancestors of these vertices. Note that $\mathscr{P}$ is lacunary of order $N$, and $\mathscr{P}$ is a tree of height $k_1 + \cdots + k_N$.

Let $\sigma: \mathscr{B}^{h(\mathscr{P})} \rightarrow \mathscr{P}
.$   Note $K_\sigma$ is the union of $2^{h(\mathscr{P})}$ parallelograms of the form $\rho_v$ indicated above, where $v$ is an element of $\mathscr{B}$ of height $h(\mathscr{P})$. 

Let $\gamma$ denote the number of  parallelograms $\rho_v$ for which there is a $w \neq v$ such that $|\rho_v \cap \rho_w| > 0\;.$ Since $\sigma$ is a sticky map and by the structure of $\mathscr{P}$, we have that, for $1 \leq l \leq N-1$,  if  $\rho_u$ and $\rho_v$ are parallelograms whose left hand sides lie in a dyadic interval of length $2^{-k_1 - k_2 - \cdots  - k_l}$ on the $y$-axis, their slopes are within $2^{-k_1 - k_2 - \cdots - k_l} \cdot 2\cdot2^{-k_{l+1}}$ of each other.   Also, all of the  $\rho_v$ have slopes within $2 \cdot 2^{-k_1}$ of each other. Accordingly we have the bound 
\begin{align} \gamma &\leq 2^{h(\mathscr{P})}8[2^{-k_1} + 2^{k_1}2^{-k_1 -k_2} + 2^{k_1 + k_2}2^{-k_1 - k_2 -k_3}  +\cdots + 2^{k_1 + \cdots + k_{N-1}}2^{-k_1 - \cdots -k_N}] \notag\\&\;\;\;= 2^{h(\mathscr{P})}8[2^{-k_1} + \cdots + 2^{-k_N}]\notag\;.\end{align}

A few remarks justifying the above inequality are in order.   Note that if $|\rho_u \cap \rho_v| \neq 0$, then $u$ and  $v$ are in different  halves of [0,1] or are in different halves of  a dyadic interval of length $2^{-k_1 - \cdots - k_j}$ for some $1 \leq j  \leq N-1$ (as otherwise $\rho_u$ and $\rho_v$ would be parallel.)     The number of parallelograms $\rho_u$ whose interior  can intersect the interior of a parallelogram $\rho_v$ where $u$ and $v$ are on opposite halves of $[0,1]$ is $2^{h(\mathscr{P})} \cdot 8 \cdot 2^{-k_1}\;.$   Similarly, given a dyadic interval in $[0,1]$ of length $2^{-k_1 - \cdots - k_{j-1}}$  (and there are $2^{k_1 + \cdots + k_{j-1}}$ of these), there are $2^{h(\mathscr{P})}  \cdot 8 \cdot 2^{-k_1 - \cdots - k_{j}}$  parallelograms $\rho_u$  that can intersect the interior of a parallelogram $\rho_v$ where $u$ and $v$ are on opposite halves of the interval.  Summing over the initial interval $[0,1]$ and all of the dyadic intervals in  $[0,1]$ with length of the form $2^{-k_1 - \cdots - k_j}$ for $1 \leq j \leq N-1$ yields the desired estimate.
\\

Since for each parallelogram $\rho_v$ we have $|\rho_v \cap ([1,2]\times \mathbb{R})| = 2^{-h(\mathscr{P})}$, we have that 
\begin{align}
&|K_\sigma \cap ( [1,2]\times \mathbb{R})| \notag \\ & \;\;\;\geq  1 -  \gamma 2^{-h(\mathscr{P})}\notag\\&\;\;\; \geq 1 - 8[2^{-k_1} + \cdots + 2^{-k_N}] \notag\\&\;\;\; > 1 - \epsilon\;,\notag\end{align}
as desired.
\end{proof}
\section{future directions}
It is highly desirable to ascertain whether the  maximal operator $M_\Omega$ is bounded on $L^p(\mathbb{R}^2)$ for all $1 < p \leq \infty$ if and only if $\Omega$ is a union of finitely many sets each of finite lacunary order.   The proof of the above theorem suggests the following model case of consideration:
\\

Let $k_1, k_2, \ldots $ be an infinite sequence of natural numbers greater than or equal to 2  such that 

$$2^{-k_1} + 2^{-k_2} + 2^{-k_3} + \cdots < \infty$$
and let, as before, $a_i$ denote the $k_i$-string $0111\cdots1$ and $b_i$ denote the $k_i$-string $1000\cdots0$.  Let $\Omega \subset [0,1]$ be the set of points with binary expansions corresponding to sequences of the form $0s_1s_2s_3\cdots s_N$, where $N$ is arbitrary in $\mathbb{N}$ and each $s_i$ is either of the form $a_i$ or $b_i$.  $\Omega$ is not finite lacunary, and hence one can not use the Sj\"ogren-Sj\"olin result in \cite{sjsj} to prove that $M_\Omega$ is bounded on $L^p(\mathbb{R}^2)$.  However, $\Omega$ also does not satisfy the separation condition found in \cite{hrs2024} that would imply that $M_\Omega$ is unbounded on $L^p(\mathbb{R}^2)$ for all $1 \leq p < \infty$.   $\Omega$ having a rather straightforward structure, however, suggests that determining whether $M_\Omega$ is bounded on $L^p(\mathbb{R}^2)$ for all $1 < p \leq \infty$ provides a good starting point for investigating the above problem.

\end{document}